\documentclass[11pt]{amsart}
\usepackage{amsmath}
\usepackage{mathrsfs}
\usepackage{dsfont}
\usepackage{tikz}
\usetikzlibrary{arrows,shapes,trees}

\theoremstyle{plain}
\newtheorem{theorem}{Theorem}[section]
\newtheorem{lemma}[theorem]{Lemma}

\theoremstyle{definition}
\newtheorem{remark}[theorem]{Remark}

\numberwithin{equation}{section}

\def\be{\begin{equation}}
\def\ee{\end{equation}}

\begin{document}

\title{Refined Asymptotics for Minimal Graphs in the Hyperbolic Space }

\author{Weiming Shen}
\address{Beijing International Center for Mathematical Research\\
Peking University\\
Beijing, 100871, China}  \email{wmshen@pku.edu.cn}
\author{Yue Wang}
\address{School of Mathematical Sciences\\
Peking University\\
Beijing, 100871, China}
\address{Beijing International Center for Mathematical Research\\
Peking University\\
Beijing, 100871, China}
\email{yuewang37@pku.edu.cn}

\begin{abstract}
We study the boundary behaviors of solutions $f$ to the Dirichlet problem
for minimal graphs in the hyperbolic space with singular asymptotic
boundaries and characterize the boundary behaviors of $f$ at the points strictly
located in the tangent cones at the singular points on the boundary. For $n=2$,
we also obtain a refined estimate of
$f$.
\end{abstract}

\thanks{Authors acknowledge the support of NSFC Grant
11571019.}
\maketitle
\section{Introduction}

Assume that $\Omega\subset \mathbb{R}^{n}$ is a bounded domain.
Lin \cite{Lin1989Invent} studied the Dirichlet problem of the form
\begin{align}\label{eq-MinGmain}
\begin{split}
\Delta f-\frac{f_{i}f_{j}}{1+|\nabla f|^{2}}f_{ij}+\frac{n}{f}&=0 \quad\text{in }\Omega,\\
f&>0 \quad\text{in }  \Omega,\\ f&=0 \quad \text{on }\partial\Omega.
\end{split}
\end{align}
Geometrically, the graph of $f$ is a minimal surface in $\mathbb H^{n+1}$ with its asymptotic boundary at
infinity given by  $\partial \Omega$. For $n=2$,
\eqref{eq-MinGmain} also appears in the study of the Chaplygin gas. See
\cite{Serre2009} for details.

The existence of a unique solution $f\in C(\bar{\Omega})\bigcap C^{\infty}(\Omega)$ to \eqref{eq-MinGmain}
 was shown in
\cite{Lin1989Invent} with the assumption that
$\Omega\subset \mathbb{R}^{n}$
is a $C^2$-domain and its boundary has nonnegative mean curvature $H_{\partial\Omega}$
with respect to the inward
normal direction of $\partial\Omega$.
Concerning the higher global regularity, Lin proved if $H_{\partial\Omega} > 0$, then
$f\in C^{1/2}(\bar{\Omega})$.
In \cite{HanShenWang1}, Han and we proved that under the condition $H_{\partial\Omega} \geq 0$,
$f\in  C^{\frac{1}{n+1}}(\bar{\Omega}).$
 Han and we also proved in \cite{HanShenWang1} that \eqref{eq-MinGmain} admits a unique solution
$f\in C^{1/2}(\bar{\Omega})\bigcap C^{\infty}(\Omega)$
under the assumption that $\Omega$ is the intersection of finitely many bounded convex $C^{2}$-domains $\Omega_i$
with $H_{\partial\Omega_i} > 0$.

Concerning asymptotic behaviors of solutions to \eqref{eq-MinGmain},
when $\Omega$ is sufficiently smooth, the expansion near the boundary of solution to the Dirichlet problem
for minimal graphs in the hyperbolic space is shown in \cite{HanJiang2014}.
When $\Omega$ is singular, Han and
we  \cite{HanShenWang2} studied the asymptotic behaviors of solution $f$ on $ \Omega\subset\mathbb R^{2}$
whose boundary are piecewise regular
with positive curvatures and derived an estimate of $f$ by comparing it with the corresponding solutions
in the intersections of interior tangent balls.

The boundary geometry has great effects on behaviors of solutions to \eqref{eq-MinGmain}.
When the boundary is regular, asymptotic behaviors are much clearer.
For example, if $\Omega$ is a bounded $C^{2,\alpha}$-domain with $H_{\partial\Omega}>0$,
for some $\alpha\in (0,1)$, then
\begin{equation}\label{eq-C^2BasicExpansion}
\bigg|\left(\frac{H_{\partial\Omega}}{2d}\right)^{\frac12}f-1\bigg|\le Cd^{\frac{\alpha}2},\end{equation}
where $d$ is the distance function to $\partial\Omega$. Another problem involving positive boundary curvatures
is discussed by Jian and Wang \cite{JianWang2013JDG}.
However, difficulties arise when
we study asymptotic behaviors of solutions $f$ in domains with singularity.
In the general case of singular boundary, it is natural to compare solutions with the
corresponding solutions in the tangent cones. This is the approach Han and the first author adopted in the study
of the Liouville
equation in \cite{HanShen1} and the Loewner-Nirenberg problem in \cite{HanShen2}.
However, for  \eqref{eq-MinGmain} in domains with singularity,
in light of \eqref{eq-C^2BasicExpansion}, we should abandon this approach,
since the boundaries of tangent cones bounded by finitely many hyperplanes  have zero mean curvature
wherever they are smooth. In a sense, we need to preserve the positivity of the boundary mean curvature.
For $n=2$,  in domains whose boundaries are piecewise regular
with positive curvatures, Han and we \cite{HanShenWang2} studied the asymptotic behaviors of $f$
to \eqref{eq-MinGmain} and proved that $f$ can be well approximated by the corresponding solutions
in the intersections of interior tangent balls.

In this paper, we continue our study of the boundary behaviors of solutions $f$ to \eqref{eq-MinGmain}
in general convex domains with singular asymptotic
boundaries. We characterize the boundary behaviors of $f$ at the points
strictly located in the tangent cones at the singular points on the boundary and
prove that $f$ at these points can be well approximated by the
corresponding solutions in tangent cones.
For $n=2$, we also obtain a refined estimate of $f$.
From the results in this paper and also the results in \cite{HanShenWang1} and \cite{HanShenWang2},
we can see the degeneracy of \eqref{eq-MinGmain} has
great effect on the boundary behaviors of the solution, which we can compare with the results in \cite{ME1988}.

Our first main theorem in this paper is the following result.

\begin{theorem}\label{main theorem1}
Let $\Omega$ be a bounded convex domain in $\mathbb R^n$ and, for some
$x_0\in\partial\Omega$ and $R>0$,
$\partial\Omega\cap B_R(x_0)$ consist of $k$ $C^{1,1}$-hypersurfaces
$S_i,\,i=1,\cdots,k,$ with the angle between any two of the tangent planes at $x_0$ less than $\pi$.
Suppose  $f\in C(\bar\Omega)\cap C^\infty (\Omega)$ is the solution of
\eqref{eq-MinGmain} in $\Omega$ and $f_V$ is the corresponding solution in the tangent cone $V$
of $\Omega$ at $x_0$. Then, for any $\delta>0$ and any $x\in \Omega$ close to $x_0$,
with $ \text{dist}(x, \partial \Omega) \geq \delta |x-x_0|$,
\begin{align}\label{n/n+1}
|f(x)-f_V(x)| \leq Cf(x)|x-x_0|,
\end{align}
where $C$ is some constant depending only on $\delta$ and the geometry of $\partial \Omega$ near $x_0.$
\end{theorem}

Inspired by results in \cite{HanShenWang2}, we
now compare solutions $f$ to \eqref{eq-MinGmain} with those
in the intersections of interior tangent balls and prove a refined estimate.

\begin{theorem}\label{main theorem2}
Let $\Omega$ be a bounded convex domain in $\mathbb R^2$ and, for some
$x_0\in\partial\Omega$ and $R>0$,
$\partial\Omega\cap B_R(x_0)$  consist of two $C^{2,\alpha}$-curves
$\sigma_1$ and $\sigma_2$
intersecting at $x_0$ with an angle $\mu\pi$, for some
constants $\alpha \in(0,1)$ and $ \mu\in (0,1)$. Assume the curvature $\kappa_i$
of $\sigma_i$ at $x_0$ is positive, $i=1,2.$
Suppose  $f\in C(\bar\Omega)\cap C^\infty (\Omega)$ is the solution of
\eqref{eq-MinGmain} in $\Omega$ and $f_*$ is the corresponding solution in
$$\Omega_{x_0,\mu,\kappa_{1},\kappa_{2}}=
B_{\frac{1}{\kappa_{1}}} \left(x_0+\frac{1}{\kappa_{1}}\nu_1\right)
\bigcap B_{\frac{1}{\kappa_{2}}} \left(x_0+\frac{1}{\kappa_{2}}\nu_2\right),$$
where $\nu_1$ and $\nu_2$ are interior unit normal vector to $\sigma_1$ and $\sigma_2$ at $x_0$,
respectively.
Then, for any $ \epsilon \in (0,\alpha)$ and $\delta>0$, there exists a constant
$ \mu({\epsilon,\alpha})>0$, such that, if $\mu \leq \mu({\epsilon,\alpha})$,
then, for any $x\in \Omega$ close to $x_0$, with $ \text{dist}(x, \partial \Omega) \geq \delta |x-x_0|$,
\begin{align}
\label{main-estimate2a} \left|f(x)-f_*(x)\right|\leq Cf(x)|x-x_0|^{1+\alpha-\epsilon},
\end{align}
where $C$ is a positive constant depending only on
$R$, $\mu$, $\alpha$, $\epsilon$, $\delta$ and
the $C^{2,\alpha}$-norms of $\sigma_1$ and $\sigma_2$ in $B_R(x_0)$.
\end{theorem}

The paper is organized as following. In Section \ref{sec-Existence},
we study the boundary behaviors of  solutions of \eqref{eq-MinGmain}
in bounded convex domains bounded by finitely many $C^{1,1}$-hypersurfaces
and prove Theorem \ref{main theorem1}.
In Section \ref{refined},
we study $f$ in domains whose boundaries are piecewise regular
with positive curvatures and prove  Theorem \ref{main theorem2}.

\section{Solutions in Convex Domains Bounded by Hypersurfaces}\label{sec-Existence}

In this section, we discuss the boundary behaviors of solutions of \eqref{eq-MinGmain}
in convex domains bounded by several $C^{1,1}$ hypersurfaces.
We prove, at points strictly located in tangent cones defined at singular points on the boundary,
the solutions $f$ are well approximated by the
corresponding solutions in these cones.

First, we discuss \eqref{eq-MinGmain}
in infinite cones and prove the existence and uniqueness of solutions of \eqref{eq-MinGmain} in infinite cones.
Since this part follows \cite{HanShenWang2} closely, we only sketch the proof.

For some constant $\mu \in (0,1)$, define
\begin{align}\label{eq-definition-cone}
\overline{V}_{\mu}&=\{(r,\theta)\mid r\in(0,\infty),\theta\in(0,\mu\pi)\}.
\end{align}
This is an infinite cone in $\mathbb R^{2}$, expressed in polar coordinates.
Then, $V_{\mu}:= \overline{V}_{\mu}\times \mathbb R^{n-2}$ is an infinite cone in $\mathbb R^{n}$.
Our goal is to find a solution $f$ to \eqref{eq-MinGmain} in $\Omega=V_{\mu}$,  whose form is given by
\begin{align}\label{f=rh}
f=rh(\theta),
\end{align}
where $(r,\theta)$ is the polar coordinates in $\mathbb R^2.$
Substituting \eqref{f=rh} in
$\eqref{eq-MinGmain}$, we have
\begin{align}
\label{po-MainEq}
\frac{h''+h}{r}-\frac{h'^{2}(h''+h)}{r(1+h^{2}+h'^{2})}+\frac{n}{rh}&=0.
\end{align}
In view of \eqref{po-MainEq}, we set the operator $ \mathcal{L}$ acting on functions $h=h(\theta)$,
$\theta\in(0,\mu\pi)$, by
\begin{align}
 \mathcal{L} h&=h(1+h^2)(h''+h)+n(1+h^{2}+h'^{2}).
\end{align}
First, we construct supersolutions of $ \mathcal{L}$.

\begin{lemma}\label{lemma-supersol}
For some constant $\mu\in (0,1)$,
there exist constants  $A>0$, $B\ge 0$, $\alpha\in[ n,+\infty)$ and $\beta\in (0,1)$
such that
\begin{align}
 \mathcal{L}\left( A(\sin\frac{\theta}{\mu})^{\frac{1}{1+\alpha}}+ B(\sin\frac{\theta}{\mu})^{\frac{1}{1+\beta}}\right)
&\leq 0\quad\text{on }(0,\mu\pi).
\end{align}
\end{lemma}

\begin{proof}
For some $\alpha>0$, set
\begin{align}\label{def-x}
\varphi(\theta)=\left(\sin\frac{\theta}{\mu}\right)^{\frac{1}{1+\alpha}}.\end{align}
By differentiating twice, we have
\begin{align*}
\varphi'=\frac{\varphi^{-\alpha}}{1+\alpha}\frac{1}{\mu}\cos\frac{\theta}{\mu},\quad\quad
\varphi''=-\frac{1}{\mu^{2}(1+\alpha)^2}\varphi-\frac{\alpha}{\mu^2(1+\alpha)^2}\varphi^{-1-2\alpha}.
\end{align*}
Then, for some positive constant $A$,
\begin{align*}
 \mathcal{L}(A\varphi)&=A^{2}\varphi(1+A^{2}\varphi^{2})
\left[(1-\frac{1}{\mu^{2}(1+\alpha)^2})\varphi-\frac{\alpha}{\mu^{2}(1+\alpha)^2}\varphi^{-2\alpha-1}\right]\\
&\qquad
+n\left[1+A^{2}\varphi^{2}+\frac{A^2}{\mu^2(1+\alpha)^2}\varphi^{-2\alpha}(1-\varphi^{2+2\alpha})\right].
\end{align*}

We first consider the case $\mu\le \frac{1}{1+n}$.
With $\alpha=n$, we have
\begin{align*}
 \mathcal{L}(A\varphi)&=A^{2}(1-\frac{1}{(1+n)^2\mu^{2}})(1+n)\varphi^{2}+n\\
&\qquad+A^{4}\varphi^{2}\left[(1-\frac{1}{(1+n)^2\mu^{2}})\varphi^{2}
-\frac{n}{(1+n)^2\mu^{2}}\varphi^{-2n}\right].\end{align*}
Hence,
\begin{align*}
 \mathcal{L}(\sqrt{(1+n)\mu} \varphi)\leq n-n\varphi^{2-2n}\leq 0.
\end{align*}

Next, we consider the case $\mu>\frac{1}{1+n}$. Fix an arbitrary constant $\alpha\in(n,+\infty).$
Set
\begin{align*}
    \psi&=(\sin\frac{\theta}{\mu})^{\frac{1}{1+\beta}},\\
    h&=A\phi+B\psi,
\end{align*}
where we take $\beta=\min\{\frac{1}{2}(\frac{1}{\mu}-1),\frac{1}{100}\},$ $A\ge 1$ to be determined, and
set $B=CA$ for a sufficiently large constant $C$ to be determined.
We can compare with the corresponding terms appearing in the proof of Lemma 2.1 in \cite{HanShenWang2}.
Then, we proceed similarly as in the proof of Lemma 2.1 in \cite{HanShenWang2} and we just point out a key difference
that, for some positive constant $\tau$, when $\sin \frac{\theta}{\mu}<\frac{1}{1+\alpha},$
\begin{align*}
 \mathcal{L}(A(\sin\frac{\theta}{\mu})^{\frac{1}{1+\alpha}}+ B(\sin\frac{\theta}{\mu})^{\frac{1}{1+\beta}})
 \leq A^{2}\frac{n-\alpha}{\mu^{2}(1+\alpha)^{2}}\varphi^{-2\alpha}
+n+CA^2\varphi^{-2\alpha+\tau},
\end{align*}
and $n-\alpha<0.$ Hence we obtain the desired result.
\end{proof}

For any $L>0,$ we define an operator $T_L$ by
\begin{align}\label{eq-T}
T_L(x_1,\cdots,x_{n+1})=\frac{L}{(x_1-L)^2+x_2^2+...+x_{n+1}^2}(L^2-|x|^2
,2Lx_2,\cdots,2Lx_{n+1}).
\end{align}
Then $T_L$ is an isometric automorphism in $\mathds{H}^{n+1},$ which maps $(L,0,\cdots,0)$ to infinity.
Restricted to $\mathbb R^{n}\times{\{x_{n+1}=0\}},$ $T_L$ is a conformal transform.
We can obtain \eqref{eq-T} by a combination of some conformal transforms in $\mathbb{R}^{n+1}.$
(See \cite{HanShenWang2}).
It is obvious that
$$\frac{2L^2 x_{n+1}}{(x_1-L)^2+\cdot\cdot\cdot+x_{n+1}^2}\rightarrow 0
\quad\text{as }x_1^2+\cdot\cdot\cdot+x_{n+1}^2\rightarrow\infty.$$

With Lemma \ref{lemma-supersol} and $T_L,$ we prove the existence and  uniqueness
of the solution of \eqref{eq-MinGmain}
in any cone $V\subseteq\mathbb R^{n}$ by following closely
the proof of Theorem 2.3 in \cite{HanShenWang2}.
In fact, any cone $V$ is contained in a cone $\overline{V}$ bounded by two hyperplanes  with a angle less than $\pi$
and the super-solution to \eqref{eq-MinGmain} on $\overline{V}$ is a upper bound for the solution
to \eqref{eq-MinGmain} on $V$ by the maximum principle. From the proof, we also conclude
that the solution in $V$ has the form
\begin{align}\label{cone-sol}
 f_{V}(x)=|x|g_V(\theta),
\end{align}
with $\theta\in \mathds{S}^{n-1}$.

\smallskip

Next, we turn our attention to \eqref{eq-MinGmain} on domains.

Let $\Omega$ be a bounded convex domain and, for some
$x_0\in\partial\Omega$ and $R>0$,
$\partial\Omega\cap B_R(x_0)$ consist of
$k$ $C^{1, 1}$-hypersurfaces $S_i$, $i=1,\cdots,k$, with the angle between any two of the tangent planes
at $x_0$ less than $\pi$. Denote by $V_{x_0}$ the
tangent cone  of $\Omega$ at $x_0.$ Then, $V_{x_0}$ is bounded by $P_i$, the tangent plane
of $S_i$ at $x_0,$ for $i=1,\cdots,k.$ Denote by $\nu_i$ the unit inner normal vector to $P_i$, $i=1,\cdots,k$.

 Assume $x_0\in\partial\Omega$ is the origin 0. By the convexity,
 $\Omega $ is a bounded Lipschitz domain and we can assume
$$\Omega\cap B_{R}(x_{0})
=\{x\in B_{R}: x_{n}>f(x')\},$$
for some Lipschitz function $f$ on $B'_{R}\subset\mathbb R^{n-1}$, with $f(0)=0$.
Then, there exists a finite circular cone
$V_{\theta_{0}}$ such that  $x_{0}$ is its vertex, the $x_n$-axis its axis,
$2\theta_{0}$ the apex angle, $h$ the height, and
$$
V_{\theta_{0}}\subseteq \overline{\Omega},\quad -V_{\theta_{0}}\subseteq \Omega^{c}.
$$
In the following, we denote by $\mu_{x_0}\pi$ the minimal angle among angles
between any two of the tangent planes at $x_0$.

For a positive constant $L$, set
\begin{equation} B_i^L=B_{\frac{L/2}{<\nu_i,e_n>}}\left(x_0+\frac{L/2}{<\nu_i,e_n>}\nu_i\right).\end{equation}
It is easy to see that
\begin{equation}\bigcap_i \partial B_i^L=\{ x_0,q\},\end{equation} where $q=x_0+Le_n$.
Note $<\nu_i, e_n>\,\,\geq  \sin\theta_0.$ Hence,
\begin{align*}
\frac{L}{2}\le \frac{L/2}{<\nu_i,e_n>}
  \le  \frac{L/2}{\sin \theta_0}.
\end{align*}
For some constant $L$ depending only on $R$ and the $C^{1,1}$-norms of $S_i$, for $i=1,\cdots,k,$
we note that each ball
$B_i^L$ is above the corresponding hypersurface $S_i$, although it is not necessarily in $\Omega$.

Now we are ready to
prove Theorem \ref{main theorem1}

\begin{proof}[Proof of Theorem \ref{main theorem1}]
Throughout the proof, we always denote by $C$
some positive constant depending only on $n$, $R$, $\theta_0$, $\delta$, $\mu_{x_0}$, $h$
and the
$C^{1,1}$-norms of hypersurfaces $S_i,\,i=1,\cdots,k$, near $x_0$. Set,
for $L$ sufficiently small,
$$\widetilde\Omega= \bigcap_iB_{i}^{2L}\subset \Omega,$$
where $B_i^{2L}$ is defined above.

Then, for $|x-x_0|$ small with $ \text{dist}(x, \partial \Omega) \geq \delta |x-x_0|$, we have
$$\textrm{dist}(x,\partial V )\geq \delta |x-x_0|,$$
and
$$\textrm{dist}(x,\partial \widetilde{\Omega} )>\frac{\delta}{2}|x-x_0|.$$

For convenience, we rotate the coordinates such that $x_n$-axis above becomes $x_1$-axis and assume
\begin{align*}
  x_0=(-L,0,\cdot\cdot\cdot,0),\quad
  q=(L,0,\cdot\cdot\cdot,0).
\end{align*}
Let $\widetilde f$ be the solution of \eqref{eq-MinGmain}
in $\widetilde\Omega$. The maximum principle implies
\begin{equation}\label{eq-LowerBound}
f\ge \widetilde f\quad\text{in }\widetilde\Omega.\end{equation}
We note that the tangent cone of $\Omega$ at $x_0$ is also the tangent
cone of $\widetilde\Omega$ at $x_0$.
We consider the map $T_{L}$ introduced in \eqref{eq-T}.
Then, $T_{L}|_{\mathbb R^n\times\{0\}}$ maps conformally $\widetilde\Omega$
to an infinite cone $\widetilde{V}$,
which conjugates to $V$, with
\begin{align}\label{same}
\widetilde{V}=V+\frac12{\overrightarrow{x_{0}q}},
\end{align}
and $T_{L}$ maps the minimal graph
$\{(x,\widetilde f(x))\}$ in $\mathds{H}^{n+1}$ to
the minimal graph $\{(y, \widetilde{f}_{\widetilde{V}}(y))\}$ in  $\mathds{H}^{n+1}$.
By \eqref{eq-T} and \eqref{cone-sol}, we have
\begin{align*}
    JT_{L}|_{x_0}=\frac{1}{2}I_{(n+1)\times (n+1)},
\end{align*}
and, for $|x-x_0|$ small and $a\in\{2,\cdot\cdot\cdot,n\},$
\begin{align*}
\bigg|y_1-\frac{1}{2}(x_1+L)\bigg|&\le
C|x-x_{0}|^2,
\\\bigg|y_a-\frac{1}{2}x_a\bigg|&\le C|x_a||x-x_{0}|,\end{align*}
and
\begin{align*}
\bigg|\widetilde{f}_{\widetilde{V}}(y)-\frac{1}{2}\widetilde f(x)\bigg|\le C\widetilde f(x)|x-x_{0}|.
\end{align*}
By (\ref{cone-sol}), when $ \text{dist}(x, \partial \Omega) \geq \delta |x-x_0|,$
$$|\nabla \widetilde{f}_{\widetilde{V}}|\leq C({\widetilde{V}},\delta)$$
and
\begin{align*}
\widetilde{f}_{\widetilde{V}}(y)&\ge \widetilde{f}_{\widetilde{V}}\left(\frac{1}{2}(x_1+L),
\frac{1}{2}x_2,\cdots,\frac{1}{2}x_n\right)
-C({\widetilde{V}},\delta)|x-x_0|^{2}\\&\ge\frac{1}{2}|x-x_0|g_{\widetilde{V}}(\theta)(1-C|x-x_0|),
\end{align*}
where we used the fact that $g_{\widetilde{V}}(\theta)\ge c$,
for some positive constant $c$ depending on ${\widetilde{V}}$ and $\delta$,
when $ \text{dist}(x, \partial \Omega) \geq \delta |x-x_0|$ and $x$ is close to $x_0$,
by noting $g_{\widetilde{V}}(\theta)>0.$
Therefore, combining (\ref{eq-LowerBound}) and the fact $g_{\widetilde{V}}=g_V$
by (\ref{same}), we have \begin{align}
   f(x)\geq |x-x_0| g_V (\theta)(1-C|x-x_0|).
\end{align} Also,
by the maximum principle, we have, for any $x\in \Omega,$
\begin{align}
   f(x)\leq f_V(x)=|x-x_0| g_V (\theta).
\end{align}
This finishes the proof.
\end{proof}

\section{Refined expansion}\label{refined}

In \cite{HanShenWang2}, we studied asymptotic behaviors of $f$ in the hyperbolic space with singular asymptotic
boundaries under the assumption that the boundaries are piecewise regular
with positive curvatures and approximated such solutions by the corresponding solutions
in the intersections of interior tangent balls up to an order $|x|^{\beta}$, with $\beta \in(0,\frac{\alpha}{2}]$.
On the other hand,  Theorem \ref{main theorem1}
demonstrates that, at points strictly located in tangent cones defined at the singular points on the boundary,
the solutions $f$ are well approximated by the
corresponding solutions in these cones up to the order $|x|$.
In light of this, we  expect that the corresponding solutions in the interior tangent balls
should provide a refined estimate over the estimate in  \cite{HanShenWang2}.

To this end, we need a localization lemma which provides more information
on the local properties of asymptotic expansions near singular boundary points up to certain orders.
Compare with Lemma 3.1 in \cite{HanShenWang2}. 

\begin{lemma}\label{localization}
Let $\Omega$ and $\Omega_*$ be two convex domains in $\mathbb R^2$ such that, for
some $x_0\in \partial\Omega$ and some $R_0\in(0,1]$,
$$\Omega\bigcap B_{R_0}(x_0)=\Omega_*\bigcap B_{R_0}(x_0),$$
and that $\partial\Omega\cap B_{R_0}(x_0)$ consists of two $C^{1,1}$-curves
$\sigma_1$, $\sigma_2$ intersecting at $x_0$,
with the angle between the tangent lines of $\sigma_1$ and $\sigma_2$
given by $\mu \pi,$ for some $\mu\in(0,1).$
Suppose that $f$ and $f_*$ are solutions of \eqref{eq-MinGmain}
for $\Omega$ and $\Omega_*$, respectively.
Then, for any $\beta>0$ and $\delta>0$, there exists a constant $\mu(\beta)$ such that,
for any $\mu \in (0,\mu(\beta)]$ and any $x\in \Omega$ close to $x_0$,
with $ \text{dist}(x, \partial \Omega) \geq \delta |x-x_0|$,
\begin{align}\label{u1-u2}
|f(x)-f_*(x)|\leq C f(x)\left(\frac{|x-x_0|}{R_{0}}\right)^{2+\beta},
\end{align}
where $C$ is a positive constant depending only
on $\mu$, $\delta$ and the $C^{1,1}$-norms of $\sigma_1$ and $\sigma_2$ in $B_{R_0}(x_0)$.
\end{lemma}

\begin{proof}
We note that the equation in \eqref{eq-MinGmain} is invariant under the scaling
$f\mapsto f(R\cdot)/R$. Without loss of generality, we assume $x_0=0$ and $R_0=1$ and prove,
for any $x\in \Omega$ close to $x_0$,
with $ \text{dist}(x, \partial \Omega) \geq \delta |x|$,
\begin{align}\label{u1-u2-equivalent}
|f(x)-f_*(x)|\leq C f(x)|x|^{2+\beta}.
\end{align}

For  any $x\in \Omega\cap B_{r_0}$ with $\text{dist}(x, \partial \Omega) \geq \delta |x|, $
we have
 \begin{align}\label{bd-f}
C_\mu |x|\ge f(x)\ge c |x|,\quad C_\mu |x|\ge f_*(x)\ge c |x|,
\end{align}
where $r_0$ and $c$ are small positive constants obtained by Theorem \ref{main theorem1} and
$C_\mu$ is a positive constant obtained by the maximum principle and \eqref{cone-sol}.
Hence, for any  $r_1\in (0,r_0)$, \eqref{u1-u2-equivalent} holds for
any $x\in \Omega\cap (B_{r_0}\setminus B_{r_1})$,
with $ \text{dist}(x, \partial \Omega) \geq \delta |x|$ by taking $C$ in \eqref{u1-u2-equivalent} large.

Let $g$ be the solution to \eqref{eq-MinGmain} in $\Omega\cap B_{1}$. By the maximum principle, we have
\begin{align}\label{eq-comparison1}f\ge g, \quad f_*\ge g\quad \text{in }\Omega\cap B_{1}.\end{align}
Write $r=|x|$. We claim, there exists a small $r_{1\mu}$ such that
\begin{align}\label{subsolution} g\ge f-Ar^{3+\beta},\quad g\ge f_*-Ar^{3+\beta}
\quad \quad\text{in }\Omega\cap B_{r_{1\mu}}.\end{align}
By combining \eqref{bd-f}, \eqref{eq-comparison1}, and \eqref{subsolution}, we have,
for any $x\in \Omega\cap  B_{r_{1\mu}}$,
with $ \text{dist}(x, \partial \Omega) \geq \delta |x|$,
$$f_*(x)\geq g(x)\geq f(x)(1-C|x|^{2+\beta}),\quad f(x)\geq g(x)\geq f_*(x)(1-C|x|^{2+\beta}).$$
Hence, we obtain \eqref{u1-u2-equivalent} for any $x\in \Omega\cap  B_{r_{1\mu}}$,
with $ \text{dist}(x, \partial \Omega) \geq \delta |x|$.

We now prove the first inequality in \eqref{subsolution}.
First, we consider the boundary condition.
Proceeding as in \cite{HanShenWang2}, we have
\begin{equation}\label{bdry1}
g\geq \bigg(1-\big(\frac{1}{r_{0_\mu}^\alpha}+2\big)r^\alpha\bigg)f \quad\text{in }\Omega\cap B_{r_{0_\mu}},
\end{equation}
and
\begin{align}\label{bdry}
\begin{split}
f&\leq C_\mu r \quad\text{in }\Omega,\\
f&=0\quad\quad\text{on }\partial\Omega,\\
C_\mu&\leq \sqrt{3\mu}\quad\text{for } \mu\leq \frac{1}{3},
\end{split}
\end{align}
where  $r_{0\mu}$ is the small positive constant defined in Lemma 3.1 in \cite{HanShenWang2}
and the subscript $\mu$ indicates its dependence on $\mu.$
In the following, we assume $\mu$ is small. Then by \eqref{bdry},
\begin{align*}
C_\mu\leq 1.
\end{align*}
Next, we require, for some small $r_{1\mu}< r_{0\mu},$
\begin{align*}
f\bigg(1-\big(\frac{1}{r_{0_\mu}^\alpha}+2\big)r^\alpha\bigg) \geq f-Ar^{3+\beta}\quad\text{on }
\Omega\cap \partial B_{r_{1\mu}}.
\end{align*}
To this end, take
\begin{align}\label{A}
  A=C_\mu\bigg(\frac{1}{r_{0_\mu}^\alpha}+2\bigg)r_{1\mu}^{\alpha-2-\beta}.
\end{align}
Combining with (\ref{bdry1}), the boundary condition is satisfied.

Next, set $$h=f-Ar^{3+\beta},$$ and
\begin{align}\label{h}
    Q(h)=\bigg(\delta_{ij}-\frac{h_ih_j}{1+|\nabla h|^2}\bigg)h_{ij}+\frac{n}{h}.
\end{align}
We will prove $Q(h)\geq 0$ in $\Omega\cap B_{r_{1\mu}}$, for the general dimension $n$.
Take $r_{1\mu}$ sufficiently small, with ${r_{1\mu}}\ll{r_{0\mu}}$. We have, for $r\leq r_{1\mu},$
\begin{align}\label{nablaAr}
    |\nabla Ar^{3+\beta}|&\leq A(3+\beta)r^{2+\beta}
    \leq(3+\beta)(\frac{1}{r_{0_\mu}^\alpha}+2)r_{1_\mu}^{\alpha}
    (\frac{r}{r_{1_\mu}})^{2+\beta}\ll 1.
\end{align}
We claim that
\begin{align}\label{claim}
   (\delta_{ij}-\frac{h_ih_j}{1+|\nabla h|^2})f_{ij} \geq
   (\delta_{ij}-\frac{f_if_j}{1+|\nabla f|^2})f_{ij}(1+C |\nabla Ar^{3+\beta}|),
\end{align}
where $C$ is a positive constant independent of $\mu.$ In fact, we have $C=2.1+1.2(n-1)$
from the proof of \eqref{claim}.
Assuming \eqref{claim}, we have, by (\ref{bdry}),
\begin{align}\label{Qh}
\begin{split}
    Q(h)&\geq (\delta_{ij}-\frac{f_if_j}{1+|\nabla f|^2})f_{ij}(1+C |\nabla Ar^{3+\beta}|)\\
    &\qquad+ (\delta_{ij}-\frac{h_ih_j}{1+|\nabla h|^2})\partial_{ij}(-Ar^{3+\beta})+\frac{n}{f-Ar^{3+\beta}}
    \\&\geq \frac{n}{f-Ar^{3+\beta}}-\frac{n}{f}(1+C |\nabla Ar^{3+\beta}|)-n(3+\beta)(2+\beta)Ar^{1+\beta}\\
    &\geq Ar^{1+\beta}\bigg(\frac{n}{C_\mu^2}(1-C_\mu (3+\beta)C)-n(3+\beta)(2+\beta)\bigg).
\end{split}
\end{align}
By (\ref{bdry}), we choose $\mu$ small so that $C_\mu$ is small. Therefore, $Q(h)\geq0.$

Now we prove \eqref{claim}. Note that
$$(\delta_{ij}-\frac{h_ih_j}{1+|\nabla h|^2})f_{ij}\,\,\text{ and }\,\,(\delta_{ij}-\frac{f_if_j}{1+|\nabla f|^2})f_{ij}$$
are invariant under constant orthogonal transforms.
Hence, in a neighborhood of any point $p\in \Omega \cap B_{r_{1_\mu}},$
by a rotation, we can assume $\nabla h(p)=h_1(p)$ and proceed to calculate at $p$ in such coordinates.
Set
$\imath,\jmath\in\{2,\cdot\cdot\cdot,n\}$ and
\begin{align}\label{aij}\begin{split}
    a_{ij}(f)&=\delta_{ij}-\frac{f_if_j}{1+|\nabla f|^2},\\a_{ij}(h)&=\delta_{ij}-\frac{h_ih_j}{1+|\nabla h|^2}.\end{split}
\end{align}
Then,
\begin{align}\label{notation}
   \begin{split}
   a_{11}(h)&=\frac{1}{1+h_1^2},\quad
     a_{\imath\imath}(h) =1, \quad
       a_{ij}(h)=0\quad\text{for } i\neq j,\\
     a_{11}(f)&=\frac{1+\sum_{\imath}f_\imath^2}{1+|\nabla f|^2}
        =\frac{1}{1+|\nabla f|^2}(1+\sum_{\imath}(\partial_\imath Ar^{3+\beta})^2),\\
     a_{\imath\imath}(f) &=1-\frac{f_\imath^2}{1+|\nabla f|^2}
     =1-\frac{(\partial_\imath Ar^{3+\beta})^2}{1+|\nabla f|^2}, \\
       |a_{ij}(f)|&=\left|-\frac{f_if_j}{1+|\nabla f|^2}\right|
        \leq \frac{|\nabla f||\nabla Ar^{3+\beta}|}{1+|\nabla f|^2}
        \quad\text{for } i\neq j,\quad i,j\in\{1,\cdot\cdot\cdot,n\},
   \end{split}
\end{align}
where we used the fact that $ h_\imath=0$ implies
\begin{align}\label{imath}
f_\imath=\partial_\imath Ar^{3+\beta}.
\end{align}
Note $a_{ii}(h),a_{ii}(f)$ are nonnegative by definition. Hence, by (\ref{imath}) again and (\ref{nablaAr}),
\begin{align*}
1+h_1^2&=1+|f_1-\partial_1(Ar^{3+\beta})|^2\\
&=1+|\nabla f|^2+|\partial_1 (Ar^{3+\beta})|^2-2f_1\partial_1 (Ar^{3+\beta})
-\sum_\imath|\partial_\imath (Ar^{3+\beta})|^2\\
&\geq (1+|\nabla f|^2)\cdot\bigg(1-\frac{f_1^2|\nabla (Ar^{3+\beta})|+|\nabla (Ar^{3+\beta})|
+\sum_\imath|\partial_\imath (Ar^{3+\beta})|^2}{1+|\nabla f|^2}\bigg)
\\&\geq (1+|\nabla f|^2)\cdot(1-|\nabla Ar^{3+\beta}|-|\nabla Ar^{3+\beta}|^2)\\
&\geq  \frac{1+|\nabla f|^2}{1+2|\nabla Ar^{3+\beta}|},
\end{align*}
and hence
\begin{align}\label{a11}
\begin{split}
a_{11}(h)=\frac{1}{1+h_1^2}&\leq  \frac{1}{1+|\nabla f|^2}(1+2|\nabla Ar^{3+\beta}|)\\
&=a_{11}(f)\frac{1+2|\nabla Ar^{3+\beta}|}{(1+\sum_{\imath}(\partial_\imath Ar^{3+\beta})^2)}
\\&\leq a_{11}(f)(1+2|\nabla Ar^{3+\beta}|),
\end{split}
\end{align}
and
\begin{align}\label{aimathimath}
a_{\imath\imath}(h)\frac{1}{1+2|\nabla Ar^{3+\beta}|}\leq a_{\imath\imath}(h)(1-|\nabla Ar^{3+\beta}|^2)
\leq a_{\imath\imath}(f).
\end{align}
Next, we consider $a_{ij}(f)f_{ij}$ for $i\neq j.$ Note that $i\neq j$ implies $i\neq 1$ or $j\neq 1$. Without loss of generality, we may assume $j\neq 1.$ By (\ref{notation}) and the concavity of $f$ from \cite{HanShenWang1}, we have
\begin{align}\label{aineqj}\begin{split}
|a_{ij}(f)f_{ij}|&
\leq \frac{|\nabla f||\nabla Ar^{3+\beta}|}{1+|\nabla f|^2}\sqrt{|f_{jj}||f_{ii}|}\\
&\leq \frac{1}{2}(\frac{|\nabla Ar^{3+\beta}|}{1+|\nabla f|^2}|f_{ii}|
+\frac{|\nabla Ar^{3+\beta}||\nabla f|^2}{1+|\nabla f|^2}|f_{jj}|)\\
&\leq\frac{1}{2}|\nabla Ar^{3+\beta}|(\frac{1}{1+|\nabla f|^2}|f_{ii}|+|f_{jj}|)\end{split}
\end{align}
Comparing the coefficients of $|f_{ii}|$ in the last inequality with $a_{ii}(f),$ we have
\begin{align}\label{aijfij}\begin{split}
  \sum_{i,j=1}^n a_{ij}(f)f_{ij}&= \sum_{i=1}^na_{ii}(f)f_{ii}+ \sum_{i\neq j} a_{ij}(f)f_{ij}\\&\leq \sum_{i=1}^na_{ii}(f)f_{ii}
 -\frac{1}{2}\sum_{i\neq j} |\nabla Ar^{3+\beta}|(a_{ii}(f)f_{ii}+a_{jj}(f)f_{jj})\frac{1}{1-|\nabla Ar^{3+\beta}|^2}
  \\&\leq \sum_{i=1}^na_{ii}(f)f_{ii}(1-1.1(n-1)|\nabla Ar^{3+\beta}|).\end{split}
\end{align}
Combining the concavity of $f,$ (\ref{nablaAr}), (\ref{a11}), (\ref{aimathimath}), and (\ref{aijfij}), we have at $p,$
\begin{align*}
  \sum_{i,j=1}^n a_{ij}(h)f_{ij}&=\sum_{i=1}^na_{ii}(h)f_{ii}\\
  &\geq \sum_{i,j=1}^n a_{ij}(f)f_{ij}\frac{1+2|\nabla Ar^{3+\beta}|}{1-1.1(n-1)|\nabla Ar^{3+\beta}|}\\
  &\geq \sum_{i,j=1}^na_{ij}(f)f_{ij}(1+(2.1+1.2(n-1))|\nabla Ar^{3+\beta}|).
\end{align*}
Therefore, we complete the proof of \eqref{claim}, with $C=2.1+1.2(n-1)$.
\end{proof}

\begin{remark}
In the above proof, we can fix a sufficiently small constant $\varepsilon_0$ independent of
$\mu,$ and then take $r_{1_\mu}=r_{0_\mu}\varepsilon_0.$ Hence, $r_{1_\mu}$ depends on $\mu$
continuously as $r_{0_\mu}$ does, which is drawn from \cite{HanShenWang2}. \end{remark}

Now we are ready to prove the refined expansions.

\begin{proof}[Proof of Theorem \ref{main theorem2}]
Fix any point $x\in \Omega$ close to $x_0$, with $ \text{dist}(x, \partial \Omega) \geq \delta |x-x_0|.$
Set
\begin{align*}
\Omega_{x_0,\mu,\kappa_{1},\kappa_{2}}=
B_{\frac{1}{\kappa_{1}}} \left(x_0+\frac{1}{\kappa_{1}}\nu_1\right)
\bigcap B_{\frac{1}{\kappa_{2}}} \left(x_0+\frac{1}{\kappa_{2}}\nu_2\right),
\end{align*}
and
\begin{align*}
\widetilde{\Omega}&=\{x'|x+(1+|x-x_0|^{1+\alpha-\epsilon})^{-1}(x'-x)\in\Omega_{x_0,\mu,\kappa_{1},\kappa_{2}}\}, \\
\widehat{\Omega}&=\{x'|x+(1-|x-x_0|^{1+\alpha-\epsilon})^{-1}(x'-x) \in\Omega_{x_0,\mu,\kappa_{1},\kappa_{2}}\}.
\end{align*}
Let $\widetilde{f}$, $\widehat{f}$, $f_*$ be the solutions of \eqref{eq-MinGmain}
for $\Omega=\widetilde{\Omega}$, $\widehat{\Omega}$, $\Omega_{x_0,\mu,\kappa_{1},\kappa_{2}}$, respectively.
Then,
$$\widetilde{f}(x)=(1+|x-x_0|^{1+\alpha-\epsilon})f_*(x),$$
and
$$\widehat{f}(x)=(1-|x-x_0|^{1+\alpha-\epsilon})f_*(x),$$

Write $\widehat{p}=x+(1-|x-x_0|^{1+\alpha-\epsilon})^{-1}(x_0-x)$.
For $|x-x_0|$ small, it is straightforward to verify
$$\widetilde\Omega'\equiv
\Omega\bigcap B_{C_0|x-x_0|^{\frac{2+\alpha-\epsilon}{2+\alpha}}}(x_0)\subset \widetilde{\Omega},$$
and
$$\widehat\Omega'\equiv
\widehat{\Omega}\bigcap B_{C_0|x-x_0|^{\frac{2+\alpha-\epsilon}{2+\alpha}}}(\widehat{p})\subset \Omega,$$
where $C_0$ is some constant depending only on
$R$, $\mu$, $\alpha$, $\epsilon$, $\delta$, $h$ and
the $C^{2,\alpha}$-norms of $\sigma_1$ and $\sigma_2$ in $B_R(x_0)$.

Let $f'$, $\widehat{f}'$ be the solutions of \eqref{eq-MinGmain} on
$\widetilde\Omega'$ and
$\widehat\Omega',$ respectively.
We choose $2+\beta=\frac{(2+\alpha)(1+\alpha-\epsilon)}{\epsilon}$ in Lemma \ref{localization}.
Then,
$$|f'(x)-f(x)|\leq
Cf(x)\left(\frac{|x-x_0|}{C_0|x-x_0|^{\frac{2+\alpha-\epsilon}{2+\alpha}}}\right)^{\frac{(2+\alpha)(1+\alpha-\epsilon)}{\epsilon}}
\leq Cf(x)|x-x_0|^{1+\alpha-\epsilon},$$
and
$$|\widehat{f}'(x)-\widehat{f}(x)|\leq
C\widehat{f}(x)
\left(\frac{|x-\widehat{p}|}{C_0|x-x_0|^{\frac{2+\alpha-\epsilon}{2+\alpha}}}\right)^{\frac{(2+\alpha)(1+\alpha-\epsilon)}{\epsilon}}
\leq C\widehat{f}(x)|x-x_0|^{1+\alpha-\epsilon},$$
where we took $R_0=C_0|x-x_0|^{\frac{2+\alpha-\epsilon}{2+\alpha}}$ in \eqref{u1-u2}.
By the maximum principle, we have
$$f'(x)\leq\widetilde{f}(x),\quad \widehat{f}'(x)\leq f(x).$$
Hence,
\begin{align*}
 \widehat{f}(x)(1-C|x-x_0|^{1+\alpha-\epsilon})&\leq f(x),\\
  f(x)(1-C|x-x_0|^{1+\alpha-\epsilon})&\leq \widetilde{f}(x).
\end{align*}
Therefore,
\begin{align}\label{eq-comparison}
f_*(x)(1-C|x-x_0|^{1+\alpha-\epsilon})\leq f(x)\leq f_*(x)
   (1+C|x-x_0|^{1+\alpha-\epsilon}).
\end{align}
This completes the proof.
\end{proof}

\end{document}